\documentclass[12pt]{amsart}

\def\co{\colon}

\usepackage{amsmath}

\makeatletter
\@namedef{subjclassname@2020}{\textup{2020} Mathematics Subject Classification}
\makeatother

\usepackage{tikz-cd}

\usepackage{graphicx}
\usepackage{mathabx}
\newtheorem{theorem}{Theorem}[section]
\newtheorem{lemma}[theorem]{Lemma}
\newtheorem{corollary}[theorem]{Corollary}

\newtheorem{proposition}[theorem]{Proposition}
\graphicspath{{./Figures/}}

\theoremstyle{definition}

\usepackage{geometry}

\usepackage{amsthm}

\makeatletter 
\newtheorem*{rep@theorem}{\rep@title}
\newcommand{\newreptheorem}[2]{%
\newenvironment{rep#1}[1]{%
 \def\rep@title{#2 \ref{##1}}%
 \begin{rep@theorem}}%
 {\end{rep@theorem}}}
\makeatother

\newreptheorem{theorem}{Theorem}

    \newgeometry{vmargin={25mm}, hmargin={20mm,20mm}}

\title{On the Mather stability theorem for smooth maps}
\date{\today}

\author{Rustam Sadykov}
\address{Kansas State University}
\email{sadykov@ksu.edu}

\begin{document}
\begin{abstract} In \cite{MaII} Mather proved that a smooth proper infinitesimally stable map is stable. This result is the key component of the Mather stability theorem \cite{MaV}, which can be reformulated as follows: a smooth proper map $f: M\to N$ is stable if and only if it is infinitesimally stable if and only if it satisfies the Mather normal crossing condition. The latter condition,  roughly speaking, means that all map germs of $f$ are stable and $f$ maps the singular strata of $f$ to $N$ in a mutually transversal manner. In this note we adapt a short argument from the book \cite{GG} to derive the Mather stability theorem presented in \cite{MaV} from the theorem in \cite{MaII}.  
\end{abstract}
\maketitle

\section{Introduction}
Let $f\co M\to N$ be a smooth map. We say that $f$ is \emph{equivalent} to a map $f'\co M\to N$ if there is a diffeomorphism $g$ of $M$ and a diffeomorphism $h$ of $N$ such that $f=hf'g^{-1}$. We denote the space of all smooth maps of $M$ into $N$ endowed with the Whitney $C^{\infty}$ topology \cite{GG} by $C^{\infty}(M, N)$.
We say that $f$ is \emph{stable} if there exists an open neighborhood $E$ of $f$ such that every map in $E$ is equivalent to $f$.  
A smooth map germ of $f\co M\to N$ at $p$ is \emph{stable}  if for every sufficiently small neighborhood $U$ of $p$ there is a neighborhood $E$ of $f$ in $C^{\infty}(M, N)$ such that for every map $f'$ in $E$ there is a point in $U$ at which the map germ of $f'$ is right-left equivalent to the map germ of $f$ at $p$ \cite[p.11]{AVGZ}.

Given a point $p\in M$, let $\Sigma_p(f)$ denote the subset of points in $M$ at which the germ of $f$ is equivalent to the germ of $f$ at $p$.  By Theorem~\ref{th:3} below, if all map germs of $f$ are stable, then $\Sigma_p(f)$ is a submanifold of $M$. 
We say that $f$ satisfies the \emph{Mather normal crossing} condition if all map germs of $f$ are stable and for any finite collection of $s>1$ distinct points $p_1, \ldots, p_s$ in $M$ with $f(p_1)=\cdots = f(p_s)=q$ the subspaces $P_i=d_{p_i}f(T\Sigma_{p_i}(f))$ of $T_qN$ are in \emph{general position}, i.e., each space $P_i$ is transverse to the intersection of the other subspaces $P_1, \ldots, \widehat{P}_i, \ldots, P_s$. 

Stability of smooth maps was extensively studied by Mather in 
a series of papers. The main result of the second paper \cite{MaII} of the series is Theorem~\ref{th:2}, see also \cite[\S V.4]{GG}. 

\begin{theorem}[Mather] \label{th:2} If $f$ is proper and infinitesimally stable (see \S\ref{s:InfSta}), then it is stable. 
\end{theorem}

Theorem~\ref{th:1} below is a reformulation (see Proposition~\ref{prop:3.3}) of the Mather stability theorem, which is one of the results in the fifth paper \cite{MaV} of the series. We will adapt a short argument from the proof of \cite[Theorem VII.6.4]{GG} to deduce Theorem~\ref{th:1} from Theorem~\ref{th:2}. 

\begin{theorem}[Mather stability theorem] \label{th:1}  Let $f\co M\to N$ be a proper map of a manifold $M$. Then the following three conditions are equivalent:
\begin{itemize} 
\item $f$ is stable, 
\item $f$ is infinitesimally stable, 
\item $f$ satisfies the Mather normal crossing condition. 
\end{itemize}
\end{theorem}

We are particularly interested in the case of Morin maps. 

\begin{corollary}\label{c:1}  A Morin map of a compact manifold is stable if and only if it satisfies the normal crossing condition. 
\end{corollary}

Corollary~\ref{c:1} is proved in the book \cite{GG} by Golubitsky and Guillemin in the case where  $\dim M =\dim N\le 4$, and  in the case of fold maps when $\dim M\ge \dim N$, see \cite[Theorem III.4.4]{GG}. 
In the case where the dimension of the manifold $N$ is $2$, Corollary~\ref{c:1} is proved in the book \cite[Theorem 4.7.6]{Wa} by Wall. Corollary~\ref{c:1} is a starting point in studying global topology of singularities of maps of low dimensional manifolds, e.g., see the book \cite{Sa} on maps of $4$-manifolds into $3$-manifolds. 

In \S\ref{s:2} and \S\ref{s:InfSta} we respectively review the notions of general position of subspaces, and infinitesimally stable maps and map germs. Readers well familiar with the book \cite{GG} may skip the review and proceed to \S\ref{s:4}, where we present the mentioned adapted short argument from \cite{GG}.

\section{Mutually transversal subspaces}\label{s:2}


Recall that for $s>1$ subspaces $P_1,\ldots, P_s$ of a vector space $Q$ are \emph{mutually transversal} or \emph{in general position} if every space $P_i$ is transverse to the intersection of the others, i.e., $Q= P_i + \cap_{j\ne i} P_j$. 

\begin{lemma}\label{l:1.1} The following conditions are equivalent:
\begin{itemize}
\item[(1)] The spaces in the collection $\{P_i\}$ are in general position.
\item[(2)] The diagonal map $\Delta^\bullet\co Q\to \oplus_i (Q/P_i)$ is surjective. 
\item[(3)] The diagonal map  $\Delta\co Q\to \oplus_i Q$ is transverse to $\oplus_i P_i$, i.e., $\oplus_i Q=\mathop\mathrm{Im} 	 	(\Delta)+ \oplus_i P_i$. 
\item[(4)] Each subcollection $\{P_{i_j}\}$ of the collection $\{P_i\}$ is in general position in $Q$.  
\item[(5)] Given $v_1,\ldots, v_s\in Q$, there is $z\in Q$ such that $v_i-z$ is in $P_i$ for all $i$. 
\end{itemize}
\end{lemma}
\begin{proof}   To show that $(1)$ implies $(2)$, suppose $Q=P_i + \cap_{j\ne i} P_j$ for each $i$.  Then every element in $\oplus_i (Q/P_i)$ can be represented by $(y_1+ P_1, \ldots, y_s+P_s)$ with $y_i\in \cap_{j\ne i} P_j$. 
Then $\Delta^\bullet(y_1+\cdots +y_s)=(y_1+ P_1, \ldots, y_s+P_s)$ since for $k\ne i$ vectors $y_k$ are in $\cap_{j\ne k} P_j\subset P_i$. Thus $\Delta^\bullet$ is surjective. Conversely, suppose $\Delta^\bullet$ is surjective. Then for any $x\in Q$ there is $p_i\in Q$ such that the $i$-th component of $\Delta^\bullet(p_i)$ is $0$, and the $j$-th component of $\Delta^\bullet(p_i)$ is $x+ P_j$ for all $j\ne i$. In other words, every element $x\in Q$ is a sum of $p_i\in P_i$ and the element $x-p_i$ in $\cap_{j\ne i}P_j$. Thus $(2)$ implies $(1)$. 
 The statements $(2)$ and $(3)$ are equivalent as well in view of the short exact sequence:
 \[
  0\longrightarrow \oplus_i P_i \longrightarrow \oplus_i Q\longrightarrow \oplus_i (Q/P_i) \longrightarrow 0. 
 \]    
 
 Suppose now that the spaces $\{P_i\}$ are in general position, i.e., $Q= P_i + \cap_{j\ne i} P_j$. Then $Q = P_i + \cap_{j\in J\setminus\{i\}} P_j$ for each subset $J$ of indices $\{1,\ldots, s\}$. 
 Thus, $(1)$ and $(4)$ are equivalent. The equivalence of $(2)$ and $(5)$ is obvious. 
\end{proof}

\section{Infinitesimally stable maps}\label{s:InfSta}

Given a smooth map $f\co M\to N$, a \emph{vector field along} $f$ is a map $w\co M\to TN$ such that $\pi_N\circ w = f$, where $\pi_N\co TN\to N$ is the canonical projection. We say that a smooth map $f\co M\to N$ is \emph{infinitesimally stable} if for every vector field $w$ along $f$, there is a vector field $u$ on $M$ and a vector field $v$ on $N$ such that  $w=df(u)+ v(f)$, where $v(f)$ stands for the composition $M\xrightarrow{f} N\xrightarrow{v} TN$. The definition of an \emph{infinitesimally stable map germ} is obtained from this definition by replacing maps and vector fields with map germs and germ vector fields respectively. We also say that the map germ of $f$ at $p$ is \emph{transverse stable} if for sufficiently big $k$ the $k$-jet extension $j^kf$ of $f$ is transverse at $p$ to the orbit of $j^kf(p)$ under the action of right-left changes of coordinates.

\begin{theorem}[Mather]\label{th:3} The following conditions on a map germ of $f\co M\to N$ at a point $p$ are equivalent:
\begin{itemize}
\item The map germ of $f$ at $p$ is stable. 
\item The map germ of $f$ at $p$ is infinitesimally stable.
\item The map germ of $f$ at $p$ is transverse stable. 
\end{itemize}
\end{theorem}
\begin{proof}  If the map germ $[f]=[f]_p$ of $f$ at $p$ is stable, then it is transverse stable. Indeed, by the Thom's jet transversality theorem  there is a map $f'$ close to $f$ in $C^\infty(M,N)$ whose jet extension is transverse in $J^k(M, N)$ to the orbit of $j^kf(p)$. If $f'$ is sufficiently close to $f$, then in view of stability of $f$ at $p$, the map germ $[f']$ of $f'$ at some point is equivalent to the map germ $[f]$ of $f$ at $p$. Therefore the jet extension of $f$ is also transverse at $p$ to the orbit of $j^kf(p)$, see \cite[Lemma V.5.7]{GG}. Every transverse stable map germ is infinitesimally stable by \cite[Theorem V.5.13]{GG}. Finally, every infinitesimally stable map germ is stable by \cite[\S 7.3]{AVGZ}. 
\end{proof}

Now a short argument in \cite[Lemma 4.6.5]{Wa}, combined with Theorem~\ref{th:3}, shows that Theorem~\ref{th:1} is indeed a reformulation of the original Mather stability theorem, see \cite[Theorem 4.1]{MaV}. 

\begin{proposition}\label{prop:3.3}
	The third condition in the original Mather stability theorem is equivalent to the third condition in Theorem~\ref{th:1}. 
\end{proposition}

\begin{proof}[Sketch of the proof.]	
	Let $f$ be a smooth proper map of $M$ into $N$, and $\pi^k\co J^k(M,N)\to M$ denote the canonical projection. The third condition in the original Mather stability theorem asserts that
	\begin{itemize}
		\item given $r\ge \dim N+1$, and $k\ge \dim N$, 
		$_rj^kf$ is transverse to every orbit in $_rJ^k(M, N)$. 
	\end{itemize}
	Here the \emph{multi-jet space} $_rJ^k(M,N)$ is the space of $r$-tuples $(z_1,\ldots, z_r)$ of $k$-jets of map germs with $\pi^k(z_i)\ne \pi^k(z_j)$ for all $i\ne j$ in $\{1,\ldots,r\}$. An \emph{orbit} in $_rJ^k(M,N)$ is the orbit of the standard action of ${\mathop\mathrm{Diff}}^kM\times {\mathop\mathrm{Diff}}^kN$ on the multi-jet space, see \cite{MaV}.
	Let $\Delta\subset M^r$ denote the subspace of tuples $p_1,\ldots, p_r$ of points in $M$ with $p_i=p_j$ for all $i\ne j$. The \emph{multi-jet section} $_rj^kf\co M^r\setminus \Delta\to\! _rJ^k(M,N)$ is defined by associating the $r$-tuple $(j^kf(p_1), \ldots, j^kf(p_r))$ of $r$-jets to a tuple $\mathbf{p}=(p_1,\ldots, p_r)$ of distinct $r$ points in $M$. A short argument (e.g., see \cite[Lemma 4.6.5]{Wa}) shows that $_rj^k f$ is transverse at $\mathbf{p}$ to the orbit of $_rj^k f(\mathbf{p})$ in the multi-jet space if and only if $j^kf$ is transverse to the orbit of $j^kf(p_i)$ at $p_i$ for each $i=1,\ldots, r$, and the subspaces $d_{p_i}f(T\Sigma_{p_i}(f))$ are in general position. Finally, by Theorem~\ref{th:3}, the jet section $j^kf$ is transverse at $p_i$ to the orbit of $j^kf(p_i)$ if and only if the map germ of $f$ at $p_i$ is stable.  
\end{proof}

Let $q$ be a point in $N$, and $S=\{p_1,\ldots, p_s\}$ a subset of $f^{-1}(q)$. We say that $f$ is \emph{infinitesimally stable at $S$} if for any germ vector fields $w_1,\ldots, w_s$ along $f$ at $p_1, \ldots, p_s$, there are germ vector fields $u_1, \ldots, u_s$ on $M$ at $p_1, \ldots, p_s$ and a germ vector field $v$ on $N$ at $q$ such that  
\begin{equation}\label{eq:stab}
w_k=df(u_k)+v(f)
\end{equation}
for each $k=1,\ldots, s$.

\begin{theorem}\label{th:inf}  Let $f\co M\to N$ be a proper smooth map. If for every point $q\in N$, the map $f$ is infinitesimally stable at every finite subset $S$ of $f^{-1}(q)$, then $f$ is infinitesimally stable. 
\end{theorem}

\begin{proof}[Sketch of the proof.] This is a version of \cite[Theorem V.1.6]{GG}. 

Recall that a point $x\in M$ is said to be a \emph{singular point} of  $f\co M\to N$ if $\mathop\mathrm{rank}(d_pf)<\min(\dim M, \dim N)$. 
If $f$ is infinitesimally stable at $S$, then the planes $P_i=df(T_{p_i}M)$ are in general position, and in particular, the set $S$ consists of at most $\dim N$ singular points of $f$. Indeed, by Lemma~\ref{l:1.1}, we need to show that for any given $v_1,\ldots, v_k\in T_qN$, there is $z\in T_qN$ such that $v_i-z$ is in $P_i$ for each $i$. Choose germ vector fields $w_i$ along $f$ at $p_i$ such that $w_i(p_i)=v_i$. For each $i=1,\ldots, s$, we have $w_i=df(u_i)+v(f)$ for  some germ vector fields $u_i$ at $p_i$ and $v$ at $q$. Then $z=v(q)$ satisfies the required property. 

Let $w$ be a given vector field along $f$. We aim to construct vector fields $u$ on $M$ and $v$ on $N$ such that $w=df(u)+v(f)$. We will assume $\dim M\ge \dim N$ as the case where $\dim M<\dim N$ is similar but easier. Let $\Sigma$ denote the set of singular points of $f$. 
Given a point $q$ in the image $f(\Sigma)$, we know that $\Sigma_q=f^{-1}(q)\cap \Sigma$ is finite. 
Therefore there are a neighborhood $W(q)$ of $q$, a neighborhood $U(q)$ of $\Sigma_q$, a 
vector field $u''$ over $U(q)$, and $v''$ over $W(q)$ such that $w|U(q)=df(u'')+v''(f)$.  
Now we may use partition of unity to construct vector fields $u'$ in a neighborhood of $\Sigma$, and $v'$ in a neighborhood of $f(\Sigma)$ such that $w=df(u')+v'(f)$ over the domain of definition. Then we extend $v'$ to a desired vector field $v$ over $N$ arbitrarily, and then extend $u'$ to a desired vector field $u$ over $M$ so that $df(u)=w-v(f)$. The latter is possible since the restriction of $f$ to $M\setminus \Sigma$ is a submersion. 
\end{proof}

\subsection{Mather's Lemma}

We recall an important result of Mather from his earlier paper \cite{MaIII}, which follows from the Malgrange preparation theorem and the Nakayama lemma. Let $f\co M\to N$ be a smooth map of manifolds of dimensions $m$ and $n$ respectively, and let $S=(p_1,\ldots, p_s)$ be a finite subset of $f^{-1}(q)$ for some point $q$.  For each $\ell$, let $f^\ell$ denote the map germ of $f$ at $p_\ell$, expressed in local coordinates $x^\ell$ centered at $p_\ell$, and local coordinates $y=(y_1,\ldots, y_n)$ centered at $q$. 
Let 
\begin{itemize}
\item $C(N)_q$ denote the ring of germs of smooth functions on $N$ at $q$,  
\item $C(M)_S=\{(g^{(1)}(x^1),\ldots, g^{(s)}(x^s))\}$ denote the product ring of germs of smooth functions on $M$ at $S$, where each $g^{(\ell)}$ is a germ of a smooth function on $M$ at $p_\ell$, 
\item $A$ denote the finitely generated $C(N)_q$-module of germ vector fields on $N$ at $q$, 
\item $B$ denote the finitely generated $C(M)_S$-module of germ vector fields on $M$ at $S$, and 
\item $C$ denote the finitely generated $C(M)_S$-module of germ vector fields along $f$ at $S$.
\end{itemize}
There are homomorphisms $tf\co B\to C$, and $wf\co A\to C$ defined by $tf(\xi)=df(\xi)$ and $wf(\eta)=\eta\circ f$. Mather's Lemma \cite[p.134]{MaIII} asserts that 
\[
   tf(B) + wf(A) + f^*(\mathfrak{m}_q)C = C
\qquad \text{implies} \qquad 
   tf(B) + wf(A) = C,
\]
where $\mathfrak{m}_q$ stands for the maximal ideal of germ functions $z$ at $q$ vanishing at $q$. Suppose that 
$f^\ell(x^\ell)=(f^\ell_1(x^\ell), \ldots, f^\ell_n(x^\ell))$ is a coordinate representation of $f^\ell$. Then a general element in the ideal $f^*(\mathfrak{m}_q)$ 
is  
\[
\left(\sum_j g^{(1)}_j(x^1)\cdot z_j\left(f^{(1)}_1(x^1), \ldots, f^{(1)}_n(x^1)\right), \ldots, \sum_j g^{(s)}_j(x^s)\cdot z_j\left(f^{(s)}_1(x^s),\ldots,  f^{(s)}_n(x^s)\right)\right), 
\]
where $z_j$ is a map germ at $q$ vanishing at $q$, and $g^\ell_j$ is a map germ at $p_\ell$.
In particular, Mather's Lemma implies that if for every germ vector field $w^\ell=\sum w_i^\ell(x^\ell)\partial_iy$ along $f$, there are germ vector fields $u^\ell=\sum u_i^\ell(x^\ell)\partial_ix^\ell$, $v=\sum v_i(f_1^\ell(x^\ell),\ldots, f_n^\ell(x^\ell))\partial_iy$,  $t^\ell=\sum t_i^\ell(x^\ell)\partial_iy$, function germs $(g_{ij}^{(1)}(x^1), \ldots, g_{ij}^{(s)}(x^s))$ in $C(M)_S$, and  function germs $f^*z_{ij}$ in $f^*\mathfrak{m}_q$ such that 
\[
df_i^\ell(u_1^\ell,\ldots, u_m^\ell) + v_i(f_1^\ell,\ldots, f_n^\ell) + \left[\sum_j  g_{ij}^\ell z_{ij}(f_1^\ell,\ldots, f_n^\ell)\right] t_i^\ell =w_i^\ell \qquad \text{for all $i$ and $\ell$},
\]
then the system
\[
df_i(u_1, \ldots, u_m) + v_i(f_1,\ldots, f_n)  =w_i^\ell \qquad\text{for all $i$ and $\ell$}
\]
can also be solved for all $w$ in terms of $u$ and $v$.

\section{Proof of Theorem~\ref{th:1}}\label{s:4}

The argument presented here is essentially an adaptation of the proof of \cite[Theorem VII.6.4]{GG}.  

Every stable map satisfies the Mather normal crossing condition by the Multijet Transversality Theorem, e.g., see \cite[Theorem 4.6.1 and Lemma 4.6.5]{Wa} or \cite[Lemma V.6.3]{GG}.  Thus, to prove Theorem~\ref{th:1}, it suffices to show that 
if a smooth map $f$ satisfies the Mather normal crossing condition, then it is infinitesimally stable at every finite subset of $f^{-1}(q)$ for every $q\in N$, as in view of Theorem~\ref{th:2} and Theorem~\ref{th:inf}, this implies that the three conditions in the Mather stability theorem are equivalent.

Suppose $f$ satisfies the Mather normal crossing condition. Let $S\subset f^{-1}(q)$ be a finite set of points for some $q\in N$ with $|S|>1$. Let $m$ and $n$ denote the dimensions of $M$ and $N$ respectively. Suppose that $m\ge n$. Then $S$ consists of the set $\{p_1,\ldots, p_s\}$ of singular points of $f$, and the set $\{p_{s+1},\ldots, p_{s+s'}\}$ of regular points of $f$. The differential $df$ of $f$ at any regular point $p_k$ in $S$ is surjective, which  implies that the infinitesimal stability equation $df(u_k)+v(f)=w_k$ in (\ref{eq:stab}) is always solvable for $u_k$ given $v$ and $w_k$. Consequently, infinitesimal stability of $f$ at $S$ is equivalent to infinitesimal stability of $f$ at the subset $\{p_1,\ldots, p_s\}$ of its singular points. Thus, in the case $m\ge n$ we may assume that $S$ consists only of singular points  $\{p_1,\ldots, p_{s}\}$.  In the case $m<n$, we let $S=\{p_1,\ldots, p_s\}$ contain both singular and regular points. 

Let $p_\ell$ be a point in $S$, and let $P_\ell$ denote the linear subspace $df(T_{p_\ell}\Sigma_{p_\ell}(f))$ of $T_qN$. Let $i_\ell$ denote the codimension of $P_\ell$.  
Then there are coordinates $(y_1^{(\ell)},\ldots, y_n^{(\ell)})$ on $N$ \emph{centered} at $q$, i.e., with $y^{(\ell)}_i(q)=0$ for all $i$, such that $P_\ell$ is given by
\[
dy_1^{(\ell)}=\cdots = dy_{i_\ell}^{(\ell)} = 0.
\] 
Put $I_1=\{1,\ldots, i_1\}$, $I_k=\{i_1+\ldots+i_{k-1}+1,\ldots, i_1+\ldots+i_k\}$ for $k=2,\ldots, s$, and $\bar{I}_\ell= \{1,\ldots, n\}\setminus I_\ell$ for $\ell=1,\ldots, s$. Since the subspaces $P_1$ and $P_2$ are transverse, the set of $i_1+i_2$ differentials 
$dy_1^{(1)}, \ldots, dy_{i_1}^{(1)}, dy_1^{(2)},\ldots, dy_{i_2}^{(2)}$  is  linearly independent, and therefore the functions $y_1^{(1)}, \ldots, y_{i_1}^{(1)}, y_1^{(2)},\ldots, y_{i_2}^{(2)}$  can be taken as the first $i_1+i_2$ coordinates near $q$.  
Similarly, 
since the subspaces $P_3$ and $P_1\cap P_2$ are transverse, we conclude that the functions $y^{(1)}_1,\ldots, y_{i_1}^{(1)}, y^{(2)}_1,\ldots, y^{(2)}_{i_2}$, $y^{(3)}_1,\ldots, y^{(3)}_{i_3}$
can be taken as the first $i_1+i_2+i_3$ coordinates. Continuing by induction (at the $k$-th step using transversality of subspaces $P_k$ and $P_1\cap \ldots \cap P_{k-1}$), we can find local coordinates $\{y_1, \ldots, y_{n}\}$ about $q$ such that each $P_\ell$ is given by the equations $dy_i=0$ for all $i\in I_\ell$. Then the restrictions of functions $\{y_i\}$ with $i\in \bar{I}_\ell$ serve as local coordinates on $P_\ell$ for $\ell=1,\ldots, s$.

Next, on $M$ there are coordinates $x^\ell=(x_1^{\ell},\ldots, x_m^{\ell})$ centered at $p_\ell$ for $\ell= 1,\ldots, s$ such that the map germ of $f$ at $p_\ell$ is given by 
\begin{equation}\label{eq:normal_form}
[y_k\circ f]_{p_\ell}=
\begin{cases}
	f_k^{(\ell)}(x_1^\ell,\ldots, x_m^\ell) & \text{if $k\in I_\ell$,}  \\
	x_{\sigma_\ell(k)}^\ell  & \text{if 
		$k\in \bar{I}_\ell$},
\end{cases}
\end{equation}
e.g., see \cite[p.161]{AVGZ}, 
where  $\sigma_\ell\co \bar{I}_\ell\to \{1,\ldots, m\}$ is an injective map of sets. Put $L_\ell=\{1,\ldots, m\}\setminus \sigma_\ell(\bar{I}_\ell)$. 
We need to show that for any germ vector fields $w^{\ell}(x^{\ell})=\sum_{i=1}^n w_i^{\ell}(x^\ell)\frac{\partial}{\partial y_i}$ along $f$ at $p_\ell$ for all $\ell= 1,\ldots, s$, there are germ vector fields $u^{\ell}(x^\ell)=\sum_{i=1}^m u_i^\ell(x^\ell)\frac{\partial}{\partial x_i^\ell}$ on $M$ at $p_\ell$, as well as a germ vector field $v(y)=\sum_{i=1}^n v_i(y)\frac{\partial}{\partial y_i}$ at $q$ on $N$ such that there are equalities of germs 
$df^{(\ell)}(u^\ell)+v(f^{(\ell)}) = w^\ell$
at $p_\ell$. In other words, we need to show that for any germ functions $w_i^\ell$ at $p_\ell$, there are germ functions $u_i^\ell$ at $p_\ell$, as well as germ functions  $v_i$ at $q$ such that
\begin{equation}\label{eq:m1}
	\left\{
	\begin{aligned}
		\frac{\partial f^{(\ell)}_k}{\partial x_1^\ell}u_1^\ell + \cdots + \frac{\partial f_k^{(\ell)}}{\partial x_{m}^\ell}u_{m}^\ell  + v_k(f_1^{(\ell)},\ldots, f_n^{(\ell)}) &= w_k^\ell  \qquad \text{if $k\in I_\ell$}, \\ 
		u_{\sigma_\ell(k)}^\ell  + v_k(f^{(\ell)}_1,\ldots, f^{(\ell)}_n) &= w_k^\ell  \qquad \text{if 
		$k\in \bar{I}_\ell$}.
	\end{aligned}
	\right.
\end{equation}

Put $c_k=v_k(0, \ldots, 0)$ for $k = 1, \ldots, n$. Solving the equations of (\ref{eq:m1}) indexed by $k\in \bar{I}_\ell$ for $u^\ell_{\sigma_\ell(k)}$ and substituting into the equations indexed by $k\in I_\ell$  yields the equivalent system:
\begin{equation}\label{eq:m3a}
	\left\{
	\begin{aligned}
		\sum_{i\in L_\ell}\frac{\partial f^{(\ell)}_k}{\partial x_i^\ell}\tilde u_i^\ell + \tilde v_k(f^{(\ell)}_1,\ldots, f^{(\ell)}_n)- \sum_{j\in \bar{I}_\ell} \frac{\partial f^{(\ell)}_{k}}{\partial x_{\sigma_\ell(j)}^\ell}\tilde v_j(f^{(\ell)}_1,\ldots, f^{(\ell)}_n) &= \tilde{w}_k^\ell  \qquad \text{if $k\in I_\ell$}, \\ 
		\tilde u_{\sigma_\ell(k)}^\ell  + \tilde v_k(f^{(\ell)}_1,\ldots, f^{(\ell)}_n) &= \tilde{w}_{k}^\ell  \qquad \text{if 
		$k\in \bar{I}_\ell$},
	\end{aligned}
	\right.
\end{equation}
where $\tilde v_k(y)=v_k(y)-c_k$ for $k\le n$, $\tilde u^\ell_i=u^\ell_i$ for $i\in L_\ell$, and $\tilde u^\ell_{\sigma_\ell(k)}=u^\ell_{\sigma_\ell(k)}+c_k$ for $k\in \bar{I}_\ell$, and
\[
\tilde w^\ell_k=
\begin{cases}
	w^\ell_k  & \text{if $k\in \bar{I}_\ell$,}  \\
	w_k^\ell - \sum_{j\in \bar{I}_\ell} \frac{\partial f^{(\ell)}_k}{\partial x_{\sigma_\ell(j)}^\ell}w_j^\ell-c_k+\sum_{i\in \bar{I}_\ell} c_i\frac{\partial f^{\ell}_k}{\partial x^\ell_{\sigma_\ell(i)}} &  \text{if $k\in I_\ell$.}
\end{cases}
\]

Since $\tilde v_k$ belongs to the maximal ideal $\frak{m}_q$ of function germs vanishing at $0$, by Mather's lemma \cite[p. 134]{MaIII}, the system (\ref{eq:m3a}) can be solved for all $\tilde{w}_i^\ell$ if and only if the \emph{reduced infinitesimal stability equations} (\ref{eq:m7}) can be solved for all $\tilde{w}_i^\ell$:  
\begin{equation}\label{eq:m7}
	\left\{
	\begin{aligned}
		\sum_{i\in L_\ell}\frac{\partial f^{(\ell)}_k}{\partial x_i^\ell}\tilde u_i^\ell + \tilde v_k(f^{(\ell)}_1,\ldots, f^{(\ell)}_n) &= \tilde{w}_k^\ell  \qquad \text{if $k\in I_\ell$}, \\ 
		\tilde u_{\sigma_\ell(k)}^\ell  + \tilde v_k(f^{(\ell)}_1,\ldots, f^{(\ell)}_n) &= \tilde{w}^\ell_k  \qquad \text{if $k\in \bar{I}_\ell$}.
	\end{aligned}
	\right.
\end{equation}
To summarize, $f$ is infinitesimally stable at $S$ if and only if for all germ functions $\tilde{w}_i^\ell$ at $p_\ell$, there are germ functions $u_i^\ell$  and $v_k$ that solve the reduced infinitesimal stability equations (\ref{eq:m7}).  Since $f$ is infinitesimally stable at $p_\ell$, for the germ functions $\tilde{w}_i^\ell(x^\ell)$ at $p_\ell$, there are germ functions $a_i^\ell(x^\ell)$ at $p_\ell$, and $b_k^\ell(y)$ at $q$ such that the reduced infinitesimal stability equations hold:

\[
\left\{
\begin{aligned}
	\sum_{i\in L_\ell}\frac{\partial f^{(\ell)}_k}{\partial x_i^\ell}a^\ell_i + b^\ell_k(f_1^{(\ell)},\ldots, f^{(\ell)}_n) &= \tilde{w}_k^\ell  \qquad \text{if $k\in I_\ell$}, \\ 
	a_{\sigma_\ell(k)}^\ell  + b_k^\ell(f^{(\ell)}_1,\ldots, f^{(\ell)}_n) &= \tilde{w}_k^\ell  \qquad \text{if $k\in \bar{I}_\ell$}.
\end{aligned}
\right.
\]

We define the germ functions $v_i$ at $q$ on $N$, and then define the germ functions $u_i$  on $M$ by 
\[
\tilde v_k=
\begin{cases}
	b_k^{\ell} & \text{if $k\in I_{\ell}$,}  \\
    0 & \text{if $k=i_1+\cdots + i_s+1, \ldots, n$},
\end{cases}
\qquad 
\tilde u_i^\ell=
\begin{cases}
	\tilde{w}_j^\ell-\tilde v_j(f^\ell)  & \text{if 
		$i=\sigma_\ell(j)$ with $j\in \bar{I}_\ell$}, \\
        a_i^\ell & \text{if $i\in L_\ell$.} 
\end{cases}
\] 

The functions $\tilde v_k$ are well-defined since the sets $I_\ell$ are disjoint and $\cup I_\ell=\{1,\ldots, i_1+\cdots +i_s\}$. The functions $\tilde{u}^\ell_i$ are well-defined since $L_\ell$ is disjoint from $\sigma_\ell(\bar{I}_\ell)$ and the union of $L_\ell$ and $\sigma_\ell(\bar{I}_\ell)$ is the set $\{1,\ldots, m\}$.
The so defined germ functions $\tilde v_i$ and $\tilde u_i^\ell$ solve the required system of equations (\ref{eq:m7}) for any germ functions $\tilde w_i^\ell$.

\end{document}